\documentclass[12pt]{article}
\usepackage{amssymb}
\usepackage{amsmath}
\usepackage{t1enc}
\usepackage[latin1]{inputenc}
\usepackage[english]{babel}
\usepackage{amsmath,amsthm}
\usepackage{amsfonts}
\usepackage{latexsym}
\usepackage[dvips]{graphicx}
\usepackage{graphicx}
\usepackage[natural]{xcolor}

\newtheorem{theorem}{Theorem}
\newtheorem{lemma}[theorem]{Lemma}

\newtheorem{definition}[theorem]{Definition}

\newtheorem{remark}[theorem]{Remark}

\input{tcilatex}
\begin{document}

\date{}
\author{Liangquan ZHANG$^{1,2}$\thanks{
Corresponding author. E-mail: xiaoquan51011@163.com. This work was supported
by Marie Curie Initial Training Network (ITN) project: \textquotedblright
Deterministic and Stochastic Controlled System and
Application\textquotedblright , FP7-PEOPLE-2007-1-1-ITN, No. 213841-2 and
National Natural Science Foundation of China Grant 10771122, Natural Science
Foundation of Shandong Province of China Grant Y2006A08 and National Basic
Research Program of China (973 Program, No. 2007CB814900). } \\
%EndAName
{1. School of mathematics, Shandong University, China.}\\
2. Laboratoire de Mathématiques, \\
Université de Bretagne Occidentale,\\
29285 Brest Cédex, France.}
\title{Stochastic Verification Theorem of Forward-Backward Controlled
Systems for Viscosity Solutions}
\maketitle

\begin{abstract}
In this paper, we investigate the controlled systems described by
forward-backward stochastic differential equations with the control
contained in drift, diffusion and generator of BSDEs. A new verification
theorem is derived within the framework of viscosity solutions without
involving any derivatives of the value functions. It is worth to pointing
out that this theorem has wider applicability than the restrictive classical
verification theorems. As a relevant problem, the optimal stochastic
feedback controls for forward-backward systems are discussed as well.
\end{abstract}


\section{Introduction}

Since the fundamental work of Pardoux \& Peng [1], the theory of BSDEs and
FBSDEs have become a powerful tool in many fields, such as mathematics
finance, optimal control, stochastic games, partial differential equations
and homogenization etc. Recently, the partially coupled FBSDEs controlled
systems have been studied in [2], [3], and [4], where the authors used the
dynamic programming principle and proved that the value function is to be
the unique viscosity solution of the H-J-B equations. In [5], the authors
investigated the existence of an optimal control for forward-backward
control systems using a verification theorem, of course, under smooth
situation. Hence, as an important part in viscosity theory, a natural
question arises: do verification theorems still hold, with the solutions of
H-J-B equations in the classical sense replaced by the ones in the viscosity
sense and the derivatives involved replaced by the super-differentials or
sub-differentials? For the deterministic and forward stochastic cases, the
answer to the above questions is \textquotedblright
positive\textquotedblright . For more details, see [6], [7], [8] and [9].

The present paper proceeds to give the answer to the above question for
forward-backward stochastic systems.

Throughout this paper, we denoted by $\mathbf{R}^n$ the space of $n$%
-dimensional Euclidean space, by $\mathbf{R}^{n\times d}$ the space the
matrices with order $n\times d$, by $\mathbf{S}^n$ the space of symmetric
matrices with order $n\times n$. $\left\langle \cdot ,\cdot \right\rangle $
and $\left| \cdot \right| $ denote the scalar product and norm in the
Euclidean space, respectively. \textbf{*} appearing in the superscripts
denoted the transpose of a matrix.

Let $T>0$ and let $\left( \Omega ,\mathcal{F},P\right) $ be a complete
probability space, equipped with a $d$-dimensional standard Brownian motion $%
\left\{ W\left( t\right) \right\} _{0\leq t\leq T}.$ For a given $s\in \left[
t,T\right] ,$ we suppose that the filtration $\left\{ \mathcal{F}%
_t^s\right\} _{s\leq t\leq T}$ is generated as the following
\begin{equation*}
\mathcal{F}_t^s=\sigma \left\{ W\left( r\right) -W\left( s\right) ;s\leq
r\leq T\right\} \vee \mathcal{N},
\end{equation*}
where $\mathcal{N}$ contains all $P$-null sets in $\mathcal{F}$. In
particular, if $s=0$ we write $\mathcal{F}_t=\mathcal{F}_t^s.$

Let $\mathcal{X}$ be a Hilbert space with the norm $\left\| \cdot \right\| _{%
\mathcal{X}},$ and $p,$ $1\leq p\leq +\infty ,$ define the set $L_{\mathcal{F%
}}^p\left( a,b;\mathcal{X}\right) =\{\left. \phi \left( \cdot \right)
=\left\{ \phi \left( t,\omega \right) :a\leq t\leq b\right\} \right| \phi
\left( \cdot \right) $ is an $\mathcal{F}_t$-adapted, $\mathcal{X}$-valued
measurable process on $\left[ a,b\right] ,$ and $\mathbf{E}\int_a^b\left\|
\phi \left( t,\omega \right) \right\| _{\mathcal{X}}^p$d$t<+\infty .\}.$

Let $U$ is a given closed set in some Euclidean space $\mathbf{R}^m$. For a
given $s\in \left[ 0,T\right] $, we denote by $\mathcal{U}_{ad}\left(
s,T\right) $ the set of $U$ -valued $\mathcal{F}_t^s$-predictable processes.
For any initial time $s\in \left[ t,T\right] $ and initial state $y\in
\mathbf{R}^d$, we consider the following stochastic control systems
\begin{equation}
\left\{
\begin{array}{l}
\text{d}X^{s,y;u}\left( t\right) =b\left( t,X^{s,y;u}\left( t\right)
,u\left( t\right) \right) \text{d}t+\sigma \left( t,X^{s,y;u}\left( t\right)
,u\left( t\right) \right) \text{d}W_t, \\
\text{d}Y^{s,y;u}\left( t\right) =-f\left( t,X^{s,y;u}\left( t\right)
,Y^{s,y;u}\left( t\right) ,Z^{s,y;u}\left( t\right) ,u\left( t\right)
\right) \text{d}t+Z^{s,y;u}\left( t\right) \text{d}W_t, \\
X^{s,y;u}\left( s\right) =x,\quad Y^{s,y;u}\left( T\right) =\Phi \left(
X^{s,y;u}\left( T\right) \right) .%
\end{array}
\right.  \tag{1.1}
\end{equation}
where
\begin{eqnarray*}
b &:&\mathbf{R}^d\times U\rightarrow \mathbf{R}^d, \\
\sigma &:&\mathbf{R}^d\times U\rightarrow \mathbf{R}^{d\times d}, \\
f &:&\left[ 0,T\right] \times \mathbf{R}^d\times \mathbf{R\times R}^d\times
U\rightarrow \mathbf{R,} \\
\Phi &:&\mathbf{R}^d\rightarrow \mathbf{R.}
\end{eqnarray*}

They satisfy the following conditions

\begin{enumerate}
\item[(H1)] $b$ and $\sigma $ are continuous in $t.$
\end{enumerate}

\begin{enumerate}
\item[(H2)] For some $L>0,$ and all $x,x^{^{\prime }}\in \mathbf{R}^d$, $%
v,v^{^{\prime }}\in U,$ a.s.
\begin{equation*}
\left| b\left( t,x,v\right) -b\left( t,x^{^{\prime }},v^{^{\prime }}\right)
\right| +\left| \sigma \left( t,x,v\right) -\sigma \left( t,x^{^{\prime
}},v^{^{\prime }}\right) \right| \leq L\left( \left| x-x^{^{\prime }}\right|
+\left| v-v^{^{\prime }}\right| \right) .
\end{equation*}
\end{enumerate}

Obviously, under the above assumptions, for any $v\left( \cdot \right) \in
\mathcal{U}_{ad}$, the first control system of (1.1) has a unique strong
solution
\begin{equation*}
\left\{ X^{s,y;u}\left( t\right) ,0\leq s\leq t\leq T\right\} .
\end{equation*}

\begin{enumerate}
\item[(H3)] $f$ and $\Phi $ are continuous in $t.$
\end{enumerate}

\begin{enumerate}
\item[(H4)] For some $L>0$, and all $x,x^{^{\prime }}\in \mathbf{R}^d$, $%
y,y^{^{\prime }}\in \mathbf{R,}z,z^{^{\prime }}\in \mathbf{R}%
^d,v,v^{^{\prime }}\in U,$ a.s.
\begin{eqnarray*}
\ \ \left| f\left( t,x,y,z,v\right) -f\left( t,x^{^{\prime }},y^{^{\prime
}},z^{^{\prime }},v^{^{\prime }}\right) \right| +\left| \Phi \left( x\right)
-\Phi \left( x^{^{\prime }}\right) \right| \\
\ \leq L\left( \left| x-x^{^{\prime }}\right| +\left| y-y^{^{\prime
}}\right| +\left| z-z^{^{\prime }}\right| +\left| v-v^{^{\prime }}\right|
\right) .
\end{eqnarray*}
\end{enumerate}

From the classical theory of BSDEs, we claim that there exists a triple $%
\left( X^{s,y;u},Y^{s,y;u},Z^{ts,y;u}\right) ,$ which is the unique solution
of the FBSDEs (1.1).

Given a control process $u\left( \cdot \right) \in \mathcal{U}_{ad}\left(
s,T\right) $ we consider the following cost functional
\begin{equation}
J\left( s,y;u\left( \cdot \right) \right) =Y^{s,y;u}\left( s\right) ,\qquad
\left( s,y\right) \in \left[ 0,T\right] \times \mathbf{R}^d,  \tag{1.2}
\end{equation}
where the process $Y^{s,y;u}$ is defined by FBSDEs (1.1). It follows from
the uniqueness of the solution of the SDEs and BSDEs that
\begin{eqnarray*}
&&Y^{s,y;u}\left( s+\delta \right) \\
&=&Y^{s+\delta ,X^{s,y;u}\left( t+\delta \right) ;u}\left( s+\delta \right)
\\
&=&J\left( t+\delta ,X^{s,y;u}\left( t+\delta \right) \right) ,\quad \text{%
a.s.}
\end{eqnarray*}

The object of the optimal control problem is to minimize the cost function $%
J\left( s,y;u\left( \cdot \right) \right) $, for a given $\left( s,y\right)
\in \left[ 0,T\right] \times \mathbf{R}^d$ , over all $u\left( \cdot \right)
\in \mathcal{U}_{ad}\left( s,T\right) .$ We denote the above problem by $%
C_{s,y} $ to recall the dependence on the initial time $s$ and the initial
state $y$. The value function is defined as
\begin{equation}
V\left( s,y\right) =\inf\limits_{u\left( \cdot \right) \in \mathcal{U}%
_{ad}\left( s,T\right) }J\left( s,y;u\left( \cdot \right) \right) .
\tag{1.3}
\end{equation}

An admissible pair $\left( X^{\star }\left( \cdot \right) ,u^{\star }\left(
\cdot \right) \right) $ is called optimal for $C_{s,y}$ if $u^{\star }\left(
\cdot \right) $ achieves the minimum of $J\left( s,y;u\left( \cdot \right)
\right) $ over $\mathcal{U}_{ad}\left( s,T\right) .$

As we have known that the verification technique plays an important role in
testing for optimality of a given admissible pair and, especially, in
constructing optimal feedback controls. Let us recall the similar classical
verification theorem as follows.

\begin{theorem}
Let $W\in C^{1,2}\left( \left[ 0,T\right] \times \mathbf{R}^d\right) $ be a
solution of the following Hamiliton-Jacobi-Bellman (H-J-B) equations:
\begin{equation}
\left\{
\begin{array}{l}
\frac \partial {\partial t}W\left( t,x\right) +H_0\left(
t,x,W,DW,D^2W\right) =0,\quad \left( t,x\right) \in \left[ 0,T\right] \times
\mathbf{R}^d, \\
W\left( T,x\right) =\Phi \left( x\right) ,\quad x\in \mathbf{R}^d.%
\end{array}
\right.  \tag{1.4}
\end{equation}
The Hamilitonian is given by
\begin{equation*}
H_0\left( t,x,W,DW,D^2W\right) =\inf\limits_{u\in U}H\left(
t,x,W,DW,D^2W,u\right) ,
\end{equation*}
where
\begin{eqnarray*}
&&H\left( t,x,\Psi ,D\Psi ,D^2\Psi ,u\right) \\
&=&\frac 12\text{tr}\left( \sigma \sigma ^{*}\left( t,x,u\right) D^2\Psi
\right) +\left\langle D\Psi ,b\left( t,x,u\right) \right\rangle \\
&&+f\left( t,x,\Psi \left( t,x\right) ,D\Psi \left( t,x\right) \cdot \sigma
\left( t,x,u\right) ,u\right) , \\
\left( t,x,u\right) &\in &\left[ 0,T\right] \times \mathbf{R}^d\times U, \\
\Psi &\in &C^{1,2}\left( \left[ 0,T\right] \times \mathbf{R}^d\right) .
\end{eqnarray*}
Here the function $b,\sigma ,f$ and $\Phi $ are supposed to satisfy
(H1)-(H4). Then

1$^{\circ })$
\begin{equation*}
W\left( s,y\right) \leq J\left( s,y;u\left( \cdot \right) \right)
\end{equation*}

for any $\left( s,y\right) \in \left[ 0,T\right] \times \mathbf{R}^d$ and $%
u\left( \cdot \right) \in \mathcal{U}_{ad}\left( s,T\right) .$

2$^{\circ })$ Supposed that a given admissible pair $\left( x^{\star }\left(
\cdot \right) ,u^{\star }\left( \cdot \right) \right) ,$ here $x^{\star
}\left( \cdot \right) =X^{\star }\left( \cdot \right) $, for the problem $%
C_{s,y}$ satisfies
\begin{eqnarray*}
&&\ \frac \partial {\partial t}W\left( t,x^{\star }\left( t\right) \right) \\
&&\ +H\left( t,x^{\star }\left( t\right) ,W\left( t,x^{\star }\left(
t\right) \right) ,DW\left( t,x^{\star }\left( t\right) \right) ,D^2W\left(
t,x^{\star }\left( t\right) \right) ,u^{\star }\left( t\right) \right) \\
\ &=&0,\quad P\text{-a.s., a.e. }t\in \left[ s,T\right] ;
\end{eqnarray*}
\begin{equation}
\tag{1.5}
\end{equation}
then $\left( x^{\star }\left( \cdot \right) ,u^{\star }\left( \cdot \right)
\right) $ is an optimal pair for the problem $C_{s,y}.$
\end{theorem}

The proof follows from Theorem 9 in Section 3 in our paper.

\begin{remark}
By H-J-B equations, (1.5) is equivalent to
the following form
\begin{eqnarray*}
\ \min\limits_{u\in U}H\left( t,x^{\star }\left( t\right) ,W\left(
t,x^{\star }\left( t\right) \right) ,DW\left( t,x^{\star }\left( t\right)
\right) ,D^2W\left( t,x^{\star }\left( t\right) \right) ,u\right) \\
\ =H\left( t,x^{\star }\left( t\right) ,W\left( t,x^{\star }\left( t\right)
\right) ,DW\left( t,x^{\star }\left( t\right) \right) ,D^2W\left( t,x^{\star
}\left( t\right) \right) ,u^{\star }\left( t\right) \right) .
\end{eqnarray*}
Then, an optimal feedback control $u^{\star }\left( t,x\right) $ can be
constructed by minimizing
\begin{equation*}
H\left( t,x,W\left( t,x\right) ,DW\left( t,x\right) ,D^2W\left( t,x\right)
,u\right)
\end{equation*}
over $u\in U.$
\end{remark}
\begin{remark} 
We claim that (1.5) is equivalent to
\begin{equation*}
W\left( s,y\right) =J\left( s,y;u^{\star }\left( \cdot \right) \right) .
\end{equation*}
Actually, we have
\begin{eqnarray*}
&&\Phi \left( X^{\star }\left( T\right) \right) -W\left( s,y\right) \\
&=&W\left( T,X^{\star }\left( T\right) \right) -W\left( s,y\right) \\
&=&\int_s^T\frac{\text{d}}{\text{d}t}W\left( t,x^{\star }\left( t\right)
\right) \text{d}t \\
&=&\int_s^T[\frac \partial {\partial t}W\left( t,x^{\star }\left( t\right)
\right) \\
&&+H\left( t,x^{\star }\left( t\right) ,W\left( t,x^{\star }\left( t\right)
\right) ,DW\left( t,x^{\star }\left( t\right) \right) ,D^2W\left( t,x^{\star
}\left( t\right) \right) ,u^{\star }\left( t\right) \right) \\
&&-f\left( t,x^{\star }\left( t\right) ,W\left( t,x\right) ,DW\left(
t,x^{\star }\left( t\right) \right) \cdot \sigma \left( t,x^{\star }\left(
t\right) ,u^{\star }\left( t\right) \right) ,u^{*}\left( t\right) \right) ]%
\text{d}t \\
&&+\int_s^TW_x\left( t,x^{\star }\left( t\right) \right) \cdot \sigma \left(
t,x^{\star }\left( t\right) ,u^{\star }\left( t\right) \right) \text{d}W_t],
\end{eqnarray*}
which implies
\begin{eqnarray*}
&&W\left( s,y\right) \\
&=&J\left( s,y;u^{\star }\left( \cdot \right) \right) +\int_s^T[\frac
\partial {\partial t}W\left( t,x^{\star }\left( t\right) \right) \\
&&+H\left( t,x^{\star }\left( t\right) ,W\left( t,x^{\star }\left( t\right)
\right) ,DW\left( t,x^{\star }\left( t\right) \right) ,D^2W\left( t,x^{\star
}\left( t\right) \right) ,u^{\star }\left( t\right) \right) ]\text{d}t
\end{eqnarray*}
\end{remark}
It is necessary to point out that in Theorem 1 we need $W\in C^{1,2}\left( %
\left[ 0,T\right] \times \mathbf{R}^{d}\right) .$ However, when we take the
verification function $W$ to be the value function $V,$ as $V$ satisfies the
HJB equations if $V\in C^{1,2}\left( \left[ 0,T\right] \times \mathbf{R}%
^{d}\right) $. Unfortunately, in general the H-J-B equations (1.4) do not
admit smooth solutions, which makes the applicability of the classical
verification theorem very restrictive and is a major deficiency in dynamic
programming theory. As we have known that the viscosity theory of nonlinear
PDEs was launched by Crandall and Lions. In this theory, all the derivatives
involved are replaced by the super-differentials and sub-differentials, and
solution in viscosity sense can be only continuous function (For more
information see in [10]). Besides, since the verification theorems can be
played primary roles in constructing optimal feedback controls, while in
many practical problems H-J-B equations do not admit smooth solutions,
hence, we want to answer the question aforementioned.

Our paper is organized as follows: In Section 2, we introduce some
preliminary results about viscosity solutions and the associated the second
order one-sided super/sub-differentials. In Section 3, a new verification
theorem in term of viscosity solutions and the super-differentials are
established. At last, we show the way to find the optimal feedback controls
in Section 4.

\section{Super-differentials, Sub-differentials, and Viscosity Solutions}

Let $Q$ be an open subset of $\mathbf{R}^n$, and $v:\overline{Q}\rightarrow
\mathbf{R}$ be a continuous function.

\begin{definition}
The second order one-sided super-differentials (resp., sub-differentials) of
$v$ at $\left( t_0,x_0\right) \in \left[ 0,T\right) \times \mathbf{R}^n,$
denoted by $D_{t+,x}^{+}v\left( t_0,x_0\right) $ (resp. $D_{t+,x}^{-}v\left(
t_0,x_0\right) $), is a set defined by
\begin{eqnarray*}
&&D_{t+,x}^{+}v\left( t_0,x_0\right) \\
&=&\{\left. \left( p,q,\Theta \right) \in \mathbf{R\times R}^d\times \mathbf{%
S}^d\right| \\
&&\overline{\lim }_{t\rightarrow t_0+,x\rightarrow x_0}\frac{v\left(
t,x\right) -v\left( t_0,x_0\right) -p\left( t-t_0\right) -\left\langle
q,x-x_0\right\rangle -\frac 12\left( x-x_0\right) ^{*}Q\left( x-x_0\right) }{%
\left| t-t_0\right| +\left| x-x_0\right| ^2} \\
&\leq &0\}\text{.}
\end{eqnarray*}
by (resp.,
\begin{eqnarray*}
&&D_{t+,x}^{-}v\left( t_0,x_0\right) \\
&=&\{\left. \left( p,q,\Theta \right) \in \mathbf{R\times R}^d\times \mathbf{%
S}^d\right| \\
&&\underline{\lim }_{t\rightarrow t_0+,x\rightarrow x_0}\frac{v\left(
t,x\right) -v\left( t_0,x_0\right) -p\left( t-t_0\right) -\left\langle
q,x-x_0\right\rangle -\frac 12\left( x-x_0\right) ^{*}Q\left( x-x_0\right) }{%
\left| t-t_0\right| +\left| x-x_0\right| ^2} \\
&\geq &0\}\text{).}
\end{eqnarray*}
\end{definition}

Let us recall the definition of a viscosity solution for (1.4) from [3] or
[4]

\begin{definition}
An continuous function $v$ on $\left[ 0,T\right] \times \mathbf{R}^n$ is
called a viscosity subsolution (resp., supersolution) of the H-J-B equations
(1.4) if
\begin{equation*}
v\left( T,x\right) \leq \Phi \left( x\right) .
\end{equation*}

and
\begin{equation}
\frac{\partial \varphi }{\partial t}\left( t_0,x_0\right) +\inf\limits_{u\in
U}\left\{ H\left( t_0,x_0,\varphi \left( t_0,x_0\right) ,D\varphi \left(
t_0,x_0\right) ,D^2\varphi \left( t_0,x_0\right) ,u\right) \right\} \geq
\left( \leq \right) 0  \tag{2.1}
\end{equation}
whenever $v-\varphi $ attains a local maximum (resp., minimum) at $\left(
t_0,x_0\right) $ in a right neighborhood of $\left( t_0,x_0\right) $ for $%
\varphi \in C^{1,2}\left( \left[ 0,T\right] \times \mathbf{R}^n\right) .$ A
function $v$ is called a viscosity solution of (1.4) if it is both a
viscosity subsolution and a supersolution of (1.4).
\end{definition}

The equivalence of Definition 4 and the Definition 5 in which derivatives of
test functions are replaced by elements of the second order one-sided sub-
and super-differentials are established with the help of a well-known result
that we present below and whose proof can be found in [13].

\begin{lemma}
Let $\left( t_0,x_0\right) \in \left[ 0,T\right] \times \mathbf{R}^n$ be
given

i) $\left( p,q,\Theta \right) \in D_{t+,x}^{+}v\left( t_0,x_0\right) $ if
and only if there exists $\varphi \in C^{1,2}\left( \left[ 0,T\right] \times
\mathbf{R}^n\right) $ satisfies
\begin{equation*}
\left( \frac{\partial \varphi }{\partial t}\left( t_0,x_0\right) ,D_x\varphi
\left( t_0,x_0\right) ,D^2\varphi \left( t_0,x_0\right) \right) =\left(
p\left( t_0,x_0\right) ,q\left( t_0,x_0\right) ,\Theta \left( t_0,x_0\right)
\right) \text{ }
\end{equation*}
and such that $v-\varphi $ achieves its maximum at $\left( t_0,x_0\right)
\in \left[ 0,T\right] \times \mathbf{R}^n$ from right side on $t$.

ii) $\left( p,q,\Theta \right) \in D_{t+,x}^{-}v\left( t_0,x_0\right) $ if
and only if there exists $\varphi \in C^{1,2}\left( \left[ 0,T\right] \times
\mathbf{R}^n\right) $ satisfies
\begin{equation*}
\left( \frac{\partial \varphi }{\partial t}\left( t_0,x_0\right) ,D_x\varphi
\left( t_0,x_0\right) ,D^2\varphi \left( t_0,x_0\right) \right) =\left(
p\left( t_0,x_0\right) ,q\left( t_0,x_0\right) ,\Theta \left( t_0,x_0\right)
\right)
\end{equation*}
and such that $v-\varphi $ achieves its minimum at $\left( t_0,x_0\right)
\in \left[ 0,T\right] \times \mathbf{R}^n$ from right side on $t$.

Moreover, if $v$ has polynomial growth, i.e., if
\begin{equation}
\left| v\left( t,x\right) \right| \leq C\left( 1+\left| x\right| ^k\right)
\text{ for some }k\geq 1,\text{ }\left( t,x\right) \in \left[ 0,T\right]
\times \mathbf{R}^n,  \tag{2.2}
\end{equation}
then $\varphi $ can be chosen so that $\varphi ,\varphi _t,D\varphi
,D^2\varphi $ satisfy (2.2) (with possibly different constants $C$).
\end{lemma}

Under the assumptions [H1]-(H4), we have the following results.

\begin{lemma}
There exists a constant $C>0$ such that, for all $0\leq t\leq T,$ $%
x,x^{^{\prime }}\in \mathbf{R}^d,$%
\begin{equation}
\left\{
\begin{array}{l}
\left| V\left( t,x\right) \leq C\left( 1+\left| x\right| \right) \right| ,
\\
\left| V\left( t,x\right) -V\left( t^{^{\prime }},x^{^{\prime }}\right)
\right| \leq C\left( \left| t-t^{^{\prime }}\right| ^{\frac 12}+\left|
x-x^{^{\prime }}\right| \right) .%
\end{array}
\right.  \tag{2.3}
\end{equation}
Moreover, $V$ is a unique solution in the class of continuous functions
which grow at most polynomially at infinity.
\end{lemma}

The proof can be seen in [2] or [4]. Then according to Definition 5 and
Lemma 6, we have the following result.

\begin{lemma}
We claim that
\begin{equation}
\inf\limits_{\left( p,q,\Theta ,u\right) \in D_{t+,x}^{+}v\left( t,x\right)
\times U}\left[ p+H\left( t,x,v,q,\Theta ,u\right) \right] \geq 0,\quad
\forall \left( t,x\right) \in \left[ 0,T\right) \times \mathbf{R}^d.
\tag{2.4}
\end{equation}
\end{lemma}

\section{Stochastic Verification Theorem for Forward-Backward Controlled
Systems}

In this section, we give the stochastic verification theorem for
Forward-Backward Controlled Systems within the framework of viscosity
solutions. Firstly, we need the following two lemmas.

\begin{lemma}
Suppose that (H1)-(H4) hold. Let $\left( s,y\right) \in \left[ 0,T\right)
\times \mathbf{R}^{d}$ be fixed and let $\left( X^{s,y;u}\left( \cdot
\right) ,u\left( \cdot \right) \right) $ be an admissible pair. Define
processes
\begin{equation*}
\left\{
\begin{array}{l}
z_{1}\left( r\right) \doteq b\left( r,X^{s,y;u}\left( r\right) ,u\left(
r\right) \right) , \\
z_{2}\left( r\right) \doteq \sigma \left( r,X^{s,y;u}\left( r\right)
,u\left( r\right) \right) \sigma ^{\ast }\left( r,X^{s,y;u}\left( r\right)
,u\left( r\right) \right) , \\
z_{3}\left( r\right) \doteq f\left( r,X^{s,y;u}\left( r\right)
,Y^{s,y;u}\left( r\right) ,Z^{s,y;u}\left( r\right) ,u\left( r\right)
\right) .%
\end{array}%
\right.
\end{equation*}%
Then
\begin{equation}
\lim\limits_{h\rightarrow 0+}\frac{1}{h}\int_{t}^{t+h}\left\vert z_{i}\left(
r\right) -z_{i}\left( t\right) \right\vert \text{d}r=0,\quad \text{a.e. }%
t\in \left[ 0,T\right] ,\text{ }i=1,2,3.  \tag{3.1}
\end{equation}
\end{lemma}

The proof can be found in [7] or [13].

\begin{lemma}
Let $g\in C\left( \left[ 0,T\right] \right) .$ Extend $g$ to $\left( -\infty
,+\infty \right) $ with $g\left( t\right) =g\left( T\right) $ for $t>T,$ and
$g\left( t\right) =g\left( 0\right) ,$ for $t<0.$ Suppose that there is a
integrable function $\rho \in L^{1}\left( 0,T;\mathbf{R}\right) $ and some $%
h_{0}>0,$ such that
\begin{equation*}
\frac{g\left( t+h\right) -g\left( t\right) }{h}\leq \rho \left( t\right)
,\quad \text{a.e. }t\in \left[ 0,T\right] ,\qquad h\leq h_{0}.
\end{equation*}%
Then
\begin{equation*}
g\left( \beta \right) -g\left( \alpha \right) \leq \int_{\alpha }^{\beta
}\limsup_{h\rightarrow 0+}\frac{g\left( t+h\right) -g\left( t\right) }{h}%
\text{d}r,\forall 0\leq \alpha \leq \beta \leq T.
\end{equation*}
\end{lemma}

\begin{proof}
Applying Fatou's Lemma, we have
\begin{eqnarray*}
\int_\alpha ^\beta \rho \left( r\right) \text{d}r &\geq &\int_\alpha ^\beta
\limsup_{h\rightarrow 0+}\frac{g\left( r+h\right) -g\left( r\right) }h\text{d%
}r \\
&\geq &\limsup_{h\rightarrow 0+}\int_\alpha ^\beta \frac{g\left( r+h\right)
-g\left( r\right) }h\text{d}r \\
&=&\limsup_{h\rightarrow 0+}\frac{\int_{\alpha +h}^{\beta +h}g\left(
r\right) \text{d}r-\int_\alpha ^\beta g\left( r\right) \text{d}r}h \\
&=&\limsup_{h\rightarrow 0+}\frac{\int_\beta ^{\beta +h}g\left( r\right)
\text{d}r-\int_\alpha ^{\alpha +h}g\left( r\right) \text{d}r}h \\
&=&g\left( \beta \right) -g\left( \alpha \right) .
\end{eqnarray*}
\end{proof}

The main result in this section is the following.

\begin{theorem}
(\textbf{Verification Theorem}) Assume that (H1)-(H4) hold. Let
\begin{equation*}
v\in C\left( \left[ 0,T\right] \times \mathbf{R}^d\right) ,
\end{equation*}
be a viscosity solution of the H-J-B equations (1.4), satisfying the
following conditions:
\begin{equation}
\left\{
\begin{array}{l}
\text{i) }v\left( t+h,x\right) -v\left( t,x\right) \leq C\left( 1+\left|
x\right| ^m\right) h,\qquad m\geq 0, \\
\quad \text{for all }x\in \mathbf{R}^d,0<t<t+h<T. \\
\text{ii) }v\text{ is semiconcave},\text{ uniformly in }t,\text{i.e}.\text{
there exists }C_0\geq 0\text{ } \\
\quad \text{such that for every }t\in \left[ 0,T\right] ,\text{ }v\left(
t,\cdot \right) -C_0\left| \cdot \right| ^2\text{is concave on }\mathbf{R}^d%
\end{array}
\right.  \tag{3.2}
\end{equation}
Then we have
\begin{equation}
v\left( s,y\right) \leq J\left( s,y;u\left( \cdot \right) \right) ,\text{
for any }\left( s,y\right) \in \left( 0,T\right] \times \mathbf{R}^d\text{
and any }u\left( \cdot \right) \in \mathcal{U}_{ad}\left( s,T\right) .
\tag{3.3}
\end{equation}
Forthurmore, let $\left( s,y\right) \in \left( 0,T\right] \times \mathbf{R}%
^d $ be fixed and let $\left( \overline{X}^{s,y;u}\left( \cdot \right) ,%
\overline{u}\left( \cdot \right) \right) $ be an admissible pair for Problem
$C_{sy}$ such that there exist a function $\varphi \in C^{1,2}\left( \left[
0,T\right] ;\mathbf{R}^d\right) $ and a triple
\begin{equation}
\left( \overline{p},\overline{q},\overline{\Theta }\right) \in \left( L_{%
\mathcal{F}_t}^2\left( s,T;\mathbf{R}\right) \times L_{\mathcal{F}%
_t}^2\left( s,T;\mathbf{R}^d\right) \times L_{\mathcal{F}_t}^2\left( s,T;%
\mathbf{S}^d\right) \right)  \tag{3.4}
\end{equation}
satisfying
\begin{equation}
\left\{
\begin{array}{l}
\left( \overline{p}\left( t\right) ,\overline{q}\left( t\right) ,\overline{%
\Theta }\left( t\right) \right) \in D_{t+,x}^{+}v\left( t,\overline{X}%
^{s,y;u}\left( t\right) \right) , \\
\left( \frac{\partial \varphi }{\partial t}\left( t,\overline{X}%
^{s,y;u}\left( t\right) \right) ,D_x\varphi \left( t,\overline{X}%
^{s,y;u}\left( t\right) \right) ,D^2\varphi \left( t,\overline{X}%
^{s,y;u}\left( t\right) \right) \right) =\left( \overline{p}\left( t\right) ,%
\overline{q}\left( t\right) ,\overline{\Theta }\left( t\right) \right) , \\
\varphi \left( t,x\right) \geq v\left( t,x\right) \quad \forall \left(
t_0,x_0\right) \neq \left( t,x\right) ,\text{ a.e. }t\in \left[ 0,T\right] ,%
\text{ }P\text{-a.s.}%
\end{array}
\right.  \tag{3.5}
\end{equation}
and
\begin{equation}
\mathbf{E}\left[ \int_s^T\left[ \overline{p}\left( t\right) +H\left( t,%
\overline{X}^{s,y;u}\left( t\right) ,\overline{\varphi }\left( t\right) ,%
\overline{p}\left( t\right) ,\overline{\Theta }\left( t\right) ,\overline{u}%
\left( t\right) \right) \right] \text{d}t\right] \leq 0,  \tag{3.6}
\end{equation}
where
\begin{equation*}
\overline{\varphi }\left( t\right) =\varphi \left( t,\overline{X}%
^{s,y;u}\left( t\right) \right) .
\end{equation*}
Then $\left( \overline{X}^{s,y;u}\left( \cdot \right) ,\overline{u}\left(
\cdot \right) \right) $ is an optimal pair for the problem $C_{sy}.$
\end{theorem}

\begin{proof} Firstly, (3.3) follows from the uniqueness of viscosity
solutions of the H-J-B equations (1.4). It remains to show that $\left(
\overline{X}^{s,y;u}\left( \cdot \right) ,\overline{u}\left( \cdot \right)
\right) $ is an optimal.

We now fix $t_0\in \left[ s,T\right] $ such that (3.4) and (3.5) hold at $%
t_0 $ and (3.1) holds at $t_0$ for
\begin{equation*}
\left\{
\begin{array}{l}
z_1\left( \cdot \right) =\overline{b}\left( \cdot \right) , \\
z_2\left( \cdot \right) =\overline{\sigma }\left( \cdot \right) \overline{%
\sigma }\left( \cdot \right) ^{*} \\
z_3\left( \cdot \right) =\overline{f}\left( \cdot \right) .%
\end{array}
\right.
\end{equation*}
We claim that the set of such points is of full measure in $\left[ s,T\right]
$ by Lemma 9. Now we fix $\omega _0\in \Omega $ such that the regular
conditional probability $\mathbf{P}\left( \left. \cdot \right| \mathcal{F}%
_{t_0}^s\right) \left( \omega _0\right) $, given $\mathcal{F}_{t_0}^s$ is
well defined. In this new probability space, the random variables
\begin{equation*}
\overline{X}^{s,y;u}\left( t_0\right) ,\overline{p}\left( t_0\right) ,%
\overline{q}\left( t_0\right) ,\overline{\Theta }\left( t_0\right)
\end{equation*}
are almost surely deterministic constants and equal to
\begin{equation*}
\overline{X}^{s,y;u}\left( t_0,\omega _0\right) ,\overline{p}\left(
t_0,\omega _0\right) ,\overline{q}\left( t_0,\omega _0\right) ,\overline{%
\Theta }\left( t_0,\omega _0\right) ,
\end{equation*}
respectively. We remark that in this probability space the Brownian motion $%
W $ is still the a standard Brownian motion although now $W\left( t_0\right)
=W\left( t_0,\omega _0\right) $ almost surely. The space is now equipped
with a new filtration $\left\{ \mathcal{F}_r^s\right\} _{s\leq r\leq T}$ and
the control process $\overline{u}\left( \cdot \right) $ is adapted to this
new filtration. For $P$-a.s. $\omega _0$ the process $\overline{X}%
^{s,y;u}\left( \cdot \right) $ is a solution of (1.1) on $\left[ t_0,T\right]
$ in $\left( \Omega ,\mathcal{F},\mathbf{P}\left( \left. \cdot \right|
\mathcal{F}_{t_0}^s\right) \left( \omega _0\right) \right) $ with the inial
condition $\overline{X}^{s,y;u}\left( t_0\right) =\overline{X}^{s,y;u}\left(
t_0,\omega _0\right) .$

Then on the probability space $\left( \Omega ,\mathcal{F},\mathbf{P}\left(
\left. \cdot \right| \mathcal{F}_{t_0}^s\right) \left( \omega _0\right)
\right) $, we are going to apply It\^o's formula to $\varphi $ on $\left[
t_0,t_0+h\right] $ for any $h>0,$%
\begin{eqnarray*}
&&\ \varphi \left( t_0+h,\overline{X}^{s,y;u}\left( t_0+h\right) \right)
-\varphi \left( t_0,\overline{X}^{s,y;u}\left( t_0\right) \right) \\
\ &=&\int_{t_0}^{t_0+h}\left[ \frac{\partial \varphi }{\partial t}\left( r,%
\overline{X}^{s,y;u}\left( r\right) \right) +\left\langle D_x\varphi \left(
r,\overline{X}^{s,y;u}\left( r\right) \right) ,\overline{b}\left( r\right)
\right\rangle \right. \\
&&\ \left. +\frac 12\text{tr}\left\{ \overline{\sigma }\left( r\right)
^{*}D_{xx}\varphi \left( r,\overline{X}^{s,y;u}\left( r\right) \right)
\overline{\sigma }\left( r\right) \right\} \right] \text{d}r \\
&&\ +\int_{t_0}^{t_0+h}\left\langle D_x\varphi \left( r,\overline{X}%
^{s,y;u}\left( r\right) \right) ,\overline{\sigma }\left( r\right)
\right\rangle \text{d}W_r.
\end{eqnarray*}
Taking conditional expectation value $\mathbf{E}^{\mathcal{F}_{t_0}^s}\left(
\cdot \right) \left( \omega _0\right) ,$ dividing both sides by $h$, and
using (3.5), we have
\begin{eqnarray*}
&&\frac 1h\mathbf{E}^{\mathcal{F}_{t_0}^s\left( \omega _0\right) }\left[
v\left( t_0+h,\overline{X}^{s,y;u}\left( t_0+h\right) \right) -v\left( t_0,%
\overline{X}^{s,y;u}\left( t_0\right) \right) \right] \\
&\leq &\frac 1h\mathbf{E}^{\mathcal{F}_{t_0}^s\left( \omega _0\right) }\left[
\varphi \left( t_0+h,\overline{X}^{s,y;u}\left( t_0+h\right) \right)
-\varphi \left( t_0,\overline{X}^{s,y;u}\left( t_0\right) \right) \right] \\
&=&\frac 1h\mathbf{E}^{^{\mathcal{F}_{t_0}^s\left( \omega _0\right)
}}\left\{ \int_{t_0}^{t_0+h}\left[ \frac{\partial \varphi }{\partial t}%
\left( r,\overline{X}^{s,y;u}\left( r\right) \right) +\left\langle
D_x\varphi \left( r,\overline{X}^{s,y;u}\left( r\right) \right) ,\overline{b}%
\left( r\right) \right\rangle \right. \right. \\
&&\left. \left. +\frac 12\text{tr}\left\{ \overline{\sigma }\left( r\right)
^{*}D_{xx}\varphi \left( r,\overline{X}^{s,y;u}\left( r\right) \right)
\overline{\sigma }\left( r\right) \right\} \right] \text{d}r\right\}
\end{eqnarray*}
\begin{equation}
\tag{3.7}
\end{equation}

Letting $h\rightarrow 0,$ and employing the similar delicate method as in
the proof of Theorem 4.1 of Gozzi et al. [12], we have
\begin{eqnarray*}
&&\frac 1h\limsup_{h\rightarrow 0+}\mathbf{E}^{\mathcal{F}_{t_0}^s\left(
\omega _0\right) }\left[ v\left( t_0+h,\overline{X}^{s,y;u}\left(
t_0+h\right) \right) -v\left( t_0,\overline{X}^{s,y;u}\left( t_0\right)
\right) \right] \\
&\leq &\frac{\partial \varphi }{\partial t}\left( t_0,\overline{X}%
^{s,y;u}\left( t_0,\omega _0\right) \right) +\left\langle D_x\varphi \left(
t_0,\overline{X}^{s,y;u}\left( t_0,\omega _0\right) \right) ,\overline{b}%
\left( t_0\right) \right\rangle \\
&&+\frac 12\text{tr}\left\{ \overline{\sigma }\left( t_0\right)
^{*}D_{xx}\varphi \left( t_0,\overline{X}^{s,y;u}\left( t_0,\omega _0\right)
\right) \overline{\sigma }\left( t_0\right) \right\} \\
&=&\overline{p}\left( t_0,\omega _0\right) +\left\langle \overline{q}\left(
t_0,\omega _0\right) ,\overline{b}\left( t_0\right) \right\rangle +\frac 12%
\text{tr}\left\{ \overline{\sigma }\left( t_0\right) ^{*}\overline{\Theta }%
\left( t_0,\omega _0\right) \overline{\sigma }\left( t_0\right) \right\}
\end{eqnarray*}
By (3.2), we know, from [12], that there exist
\begin{equation*}
\rho \in L^1\left( t_0,T;\mathbf{R}\right) \text{ and }\rho _1\in L^1\left(
\Omega ;\mathbf{R}\right)
\end{equation*}
such that
\begin{equation}
\mathbf{E}\left[ \frac 1h\left[ v\left( t+h,\overline{X}^{s,y;u}\left(
t+h\right) \right) -v\left( t,\overline{X}^{s,y;u}\left( t\right) \right) %
\right] \right] \leq \rho \left( t\right) ,\text{ for }h\leq h_0\text{, for
some }h_0>0,  \tag{3.8}
\end{equation}
and
\begin{eqnarray*}
&&\mathbf{E}^{\mathcal{F}_{t_0}^s\left( \omega _0\right) }\left[ \frac 1h%
\left[ v\left( t+h,\overline{X}^{s,y;u}\left( t+h\right) \right) -v\left( t,%
\overline{X}^{s,y;u}\left( t\right) \right) \right] \right] \\
&\leq &\rho _1\left( \omega _0\right) ,\text{ for }h\leq h_0\text{, for some
}h_0>0.
\end{eqnarray*}
\begin{equation}
\tag{3.9}
\end{equation}
holds, respectively.
By virtue of Fatou's Lemma, noting (3.9), we obtain
\begin{eqnarray*}
&&\limsup_{h\rightarrow 0+}\frac 1h\mathbf{E}\left[ v\left( t_0+h,\overline{X%
}^{s,y;u}\left( t_0+h\right) \right) -v\left( t_0,\overline{X}^{s,y;u}\left(
t_0\right) \right) \right] \\
&=&\limsup_{h\rightarrow 0+}\frac 1h\mathbf{E}\left[ \mathbf{E}^{\mathcal{F}%
_{t_0}^s\left( \omega _0\right) }\left\{ v\left( t_0+h,\overline{X}%
^{s,y;u}\left( t_0+h\right) \right) -v\left( t_0,\overline{X}^{s,y;u}\left(
t_0\right) \right) \right\} \right] \\
&\leq &\mathbf{E}\left[ \limsup_{h\rightarrow 0+}\frac 1h\mathbf{E}^{%
\mathcal{F}_{t_0}^s\left( \omega _0\right) }\left\{ v\left( t_0+h,\overline{X%
}^{s,y;u}\left( t_0+h\right) \right) -v\left( t_0,\overline{X}^{s,y;u}\left(
t_0\right) \right) \right\} \right] \\
&\leq &\mathbf{E}\left[ \overline{p}\left( t_0\right) +\left\langle
\overline{q}\left( t_0\right) ,\overline{b}\left( t_0\right) \right\rangle
+\frac 12\text{tr}\left\{ \overline{\sigma }\left( t_0\right) ^{*}\overline{%
\Theta }\left( t_0\right) \overline{\sigma }\left( t_0\right) \right\} %
\right] ,
\end{eqnarray*}
\begin{equation}
\tag{3.10}
\end{equation}
for a.e. $t_0\in \left[ s,T\right] .$ Then the rest of the proof goes
exactly as in [11]. We apply Lemma 10 to
\begin{equation*}
g\left( t\right) =\mathbf{E}\left[ v\left( t,\overline{X}^{s,y;u}\left(
t\right) \right) \right] ,
\end{equation*}
using (3.8), (3.6) and (3.10) to get
\begin{eqnarray*}
&&\mathbf{E}\left[ v\left( T,\overline{X}^{s,y;u}\left( T\right) \right)
-v\left( s,y\right) \right] \\
&\leq &\mathbf{E}\left\{ \int_s^T\left[ \overline{p}\left( t\right)
+\left\langle \overline{q}\left( t\right) ,\overline{b}\left( t\right)
\right\rangle +\frac 12\text{tr}\left[ \overline{\sigma }\left( t\right) ^{*}%
\overline{\Theta }\left( t\right) \overline{\sigma }\left( t\right) \right]
\text{d}t\right] \right\} \\
&\leq &-\mathbf{E}\left[ \int_s^T\overline{f}\left( t\right) \text{d}t\right]
.
\end{eqnarray*}
From this we claim that
\begin{eqnarray*}
v\left( s,y\right) &\geq &\mathbf{E}\left[ v\left( T,\overline{X}%
^{s,y;u}\left( T\right) \right) +\int_s^T\overline{f}\left( t\right) \text{d}%
t\right] \\
&=&\mathbf{E}\left[ \Phi \left( \overline{X}^{s,y;u}\left( T\right) \right)
+\int_s^T\overline{f}\left( t\right) \text{d}t\right] .
\end{eqnarray*}
Thus, combining the above with the first assertion (3.3), we prove the $%
\left( \overline{X}^{s,y;u}\left( \cdot \right) ,\overline{u}\left( \cdot
\right) \right) $ is an optimal pair. The proof is complete.
\end{proof}

\begin{remark}
The condition (3.6) is just equivalent to the following:
\begin{eqnarray*}
\overline{p}\left( t\right) &=&\min\limits_{u\in U}H\left( t,\overline{X}%
^{s,y;u}\left( t\right) ,\overline{\varphi }\left( t\right) ,\overline{q}%
\left( t\right) ,\overline{\Theta }\left( t\right) ,u\right) \\
&=&H\left( t,\overline{X}^{s,y;u}\left( t\right) ,\overline{\varphi }\left(
t\right) ,\overline{q}\left( t\right) ,\overline{\Theta }\left( t\right) ,%
\overline{u}\left( t\right) \right) , \\
\text{a.e. }t &\in &\left[ s,T\right] ,\text{ }P\text{-a.s.,}
\end{eqnarray*}
\begin{equation}
\tag{3.11}
\end{equation}
where $\overline{\varphi }\left( t\right) $ is defined in Theorem 11. This is
easily seen by recalling the fact that $v$ is the viscosity solution of
(1.4):
\begin{equation*}
\overline{p}\left( t\right) +\min\limits_{u\in U}H\left( t,\overline{X}%
^{s,y;u}\left( t\right) ,\overline{\varphi }\left( t\right) ,\overline{q}%
\left( t\right) ,\overline{\Theta }\left( t\right) ,u\right) \geq 0,
\end{equation*}
which yields (3.11) under (3.6).
\end{remark}

\section{Optimal Feedback Controls}

In this section, we describe the method to construct optimal feedback
controls by the verification Theorem 11 obtained. First, let us recall the
definition of admissible feedback controls.

\begin{definition}
A measurable function $\mathbf{u}$ from $\left[ 0,T\right] \times \mathbf{R}%
^d$ to $U$ is called an admissible feedback control if for any $\left(
s,y\right) \in \left[ 0,T\right) \times \mathbf{R}^d$ there is a weak
solution $X^{s,y;u}\left( \cdot \right) $ of the following SDEs:
\begin{equation}
\left\{
\begin{array}{l}
\text{d}X^{s,y;u}\left( t\right) =b\left( t,X^{s,y;u}\left( t\right) ,%
\mathbf{u}\left( t\right) \right) \text{d}t+\sigma \left( t,X^{s,y;u}\left(
t\right) ,\mathbf{u}\left( t\right) \right) \text{d}W\left( t\right) , \\
\text{d}Y^{s,y;u}\left( t\right) =-f\left( t,X^{s,y;u}\left( t\right)
,Y^{s,y;u}\left( t\right) ,\mathbf{u}\left( t\right) \right) \text{d}t+\text{%
d}M^{s,y;u}\left( t\right) , \\
X^{s,y;u}\left( s\right) =x,\quad Y^{s,y;u}\left( T\right) =\Phi \left(
X^{s,y;u}\left( T\right) \right) ,%
\end{array}
\right.  \tag{4.1}
\end{equation}
where $M^{s,y;u}$ is an $\mathbf{R}$-valued $\mathbb{F}^{s,y;u}$-adapted
right continuous and left limit martingale vanishing in $t=0$ which is
orthogonal to the driving Brownian motion $W.$ Here $\mathbb{F}%
^{s,y;u}=\left( \mathcal{F}_t^{X^{s,y;u}}\right) _{t\in \left[ s,T\right] }$
is the smallest filtration and generated by $X^{s,y;u}$, which is such that $%
X^{s,y;u}$ is $\mathbb{F}^{s,y;u}$-adapted. Obviously, $M^{s,y;u}$ is a part
of the solution of BSDEs of (4.1). Simultaneously, we suppose that $f$
satisfies the Lipschitz condition.
\begin{eqnarray*}
\left| f\left( t,x,y,u\right) -f\left( t,x^{^{\prime }},y^{^{\prime
}},u^{^{\prime }}\right) \right| \leq L\left( \left| x-x^{^{\prime }}\right|
+\left| y-y^{^{\prime }}\right| +\left| u-u^{^{\prime }}\right| \right) \\
x,x^{^{\prime }}\in \mathbf{R}^d,y,y^{^{\prime }}\in \mathbf{R,}\text{ }%
u,u^{^{\prime }}\in U.
\end{eqnarray*}

An admissible feedback control $\mathbf{u}^{\star }$ is called optimal if $%
\left( X^{\star }\left( \cdot ;s,y\right) ,\mathbf{u}^{\star }\left( \cdot
,X^{\star }\left( \cdot ;s,y\right) \right) \right) $ is optimal for the
problem $C_{s,y}$ for each $\left( s,y\right) $ is a solution of (4.1)
corresponding to $\mathbf{u}^{\star }.$
\end{definition}

\begin{theorem}
Let $\mathbf{u}^{\star }$ be an admissible feedback control and $p^{\star
},q^{\star },$ and $\Theta ^{\star }$ be measurable functions satisfying
\begin{equation}
\left( p^{\star }\left( t,x\right) ,q^{\star }\left( t,x\right) ,\Theta
\left( t,x\right) \right) \in D_{t+,x}^{+}V\left( t,x\right)  \tag{4.2}
\end{equation}
for all $\left( t,x\right) \in \left[ 0,T\right] \times \mathbf{R}^d.$ If
\begin{eqnarray*}
&&p^{\star }\left( t,x\right) +H\left( t,x,V\left( t,x\right) ,q^{\star
}\left( t,x\right) ,\Theta ^{\star }\left( t,x\right) ,\mathbf{u}^{\star
}\left( t,x\right) \right) \\
&=&\inf\limits_{\left( p,q,\Theta ,u\right) \in D_{t+,x}^{+}V\left(
t,x\right) \times U}\left[ p+H\left( t,x,V\left( t,x\right) ,q,\Theta
,u\right) \right] \\
&=&0
\end{eqnarray*}
\begin{equation}
\tag{4.3}
\end{equation}
for all $\left( t,x\right) \in \left[ 0,T\right] \times \mathbf{R}^d,$ then $%
\mathbf{u}^{\star }$ is optimal.
\end{theorem}

\textbf{$Proof$ }From Theorem 11, we get the desired result. \quad $\Box$

\begin{remark}
Actually, it is fairly easy to check that in Eq.(4.1), $Y^{s,y;u}\left(
\cdot \right) $ is determined by $\left( X^{s,y;u}\left( \cdot \right)
,u\left( \cdot \right) \right) .$ Hence, we need to investigate the
conditions imposed in Theorem 11 to ensure the existence and uniqueness of $%
X^{s,y;u}\left( \cdot \right) $ in law and the measurability of the
multifunctions $\left( t,x\right) \rightarrow D_{t+,x}^{+}V\left( t,x\right)
$ to obtain $\left( p^{\star }\left( t,x\right) ,q^{\star }\left( t,x\right)
,\Theta \left( t,x\right) \right) \in D_{t+,x}^{+}V\left( t,x\right) $ that
minimizes (4.3) by virtue of the celebrated Filippov's Lemma. The rest parts
we can get from [7] or [13].
\end{remark}

\textbf{Acknowledgments. }The author would like to thank the anonymous
referee for the careful reading of the manuscript and helpful suggestions.

\end{document}